\documentclass[11pt,reqno]{amsart}
\usepackage{amsmath, amssymb, amsthm}
\usepackage{url}
\usepackage[breaklinks]{hyperref}
\usepackage{lipsum}
\usepackage{curve2e}
\usepackage{microtype}
\usepackage[utf8]{inputenx}
\definecolor{gray}{gray}{0.4}

\allowdisplaybreaks
\usepackage{mathptmx}

\usepackage{amssymb,amsthm}
\usepackage{dsfont}
\usepackage{mathrsfs}
\usepackage{a4wide}
\usepackage{enumerate}
\usepackage{tfrupee}
\renewcommand{\a}{\alpha}
\renewcommand{\b}{\beta}

\renewcommand{\(}{\left\(}
\renewcommand{\)}{\right\)}
\renewcommand{\[}{\left\[}
\renewcommand{\]}{\right\]}
\renewcommand{\i}{\infty}
\numberwithin{equation}{section}
\theoremstyle{plain}
\newtheorem{theorem}{Theorem}[section]

\newtheorem{corollary}[theorem]{Corollary}

\makeatletter
\def\proof{\@ifnextchar[{\@oproof}{\@nproof}}
\def\@oproof[#1][#2]{\trivlist\item[\hskip\labelsep\textit{#2 Proof of\
		#1.}~]\ignorespaces}
\def\@nproof{\trivlist\item[\hskip\labelsep\textit{Proof.}~]\ignorespaces}

\makeatother

\begin{document}
\title[A finite analogue of Fine's function $F(a,b;t)$]{A finite analogue of Fine's function $F(a,b;t)$} 
	
	\author{Ritika Goel}
	\address{Discipline of Mathematics, Indian Institute of Technology Gandhinagar, Palaj, Gandhinagar 382355, Gujarat, India} 
	\email{ritikag@iitgn.ac.in}
	\begin{abstract}
	We initiate a systematic development of $F_N(a, b; t)$, a finite analogue of Fine's function $F(a, b; t)$. Our results are transformations between $F_N(a, b; t)$ and $F_N(aq^{\ell}, bq^{m}; tq^{n})$, where $\ell,m$ and $n$ take the values $0$ or $1$. 
	\end{abstract}
\maketitle
\section{Introduction}\label{intro}
In \cite{fine}, Nathan J.~Fine extensively studied the function $F(a, b; t)$ which he defined by
\begin{equation}
F(a, b; t):=\sum_{n=0}^{\infty}\frac{(aq)_n}{(bq)_n}t^n,
\end{equation}
where, here as well as in the rest of the paper, the  following standard $q$-series notation is used:
\begin{align*}
	(A)_0 &:=(A;q)_0 =1, \qquad \\
	(A)_n &:=(A;q)_n  = (1-A)(1-Aq)\cdots(1-Aq^{n-1}),
	\qquad n \geq 1, \\
	(A)_{\infty} &:=(A;q)_{\i}  = \lim_{n\to\i}(A;q)_n, \qquad |q|<1,\\
	(A_1,A_2,\cdots,A_m;q)_n&:=(A_1;q)_n(A_2;q)_n\cdots(A_m;q)_n.
\end{align*}
As shown by Fine, $F(a, b; t)$ satisfies nice properties galore and which have applications in basic hypergeometric series, theory of partitions, modular forms and mock theta functions. The first chapter of Fine's book \cite{fine} is devoted to establishing functional equations satisfied by $F(a, b; t)$, for example, \cite[p.~2, Equation (4.1)]{fine}
\begin{equation}\label{f4.1}
F(a, b; t)=\frac{1-atq}{1-t}+\frac{(1-aq)(b-atq)}{(1-bq)(1-t)}tqF(aq, bq; tq).
\end{equation}

There have been attempts to generalize the theory of $F(a, b; t)$ by considering the more general function $\displaystyle\sum\limits_{n=0}^{\infty}\frac{(aq)_n(bq)_n}{(cq)_n(dq)_n}t^n$ although they have not been quite successful, see, for example, \cite{gupta}.

Another direction is to study finite analogues of $F(a, b; t)$. Andrews and Bell \cite{andrewsbell} studied the finite analogue
\begin{equation}
F(a, b, t, N):=\sum_{n=0}^{N}\frac{(aq)_n}{(bq)_n}t^n.
\end{equation}
They constructed another function using $F(a, b, t, N)$ , namely, $S_{M, N}(a, b, t)$, which specializes to the series $S_{M, N}$ used by Euler in his unpublished proof of his famous pentagonal number theorem. One of the properties Andrews and Bell proved in their paper \cite[Lemma 2.1]{andrewsbell} is that for $N\geq1$, 
\begin{align}\label{f4.1finite}
F(a, b, t, N)=\frac{1-atq}{1-t}+\frac{(1-aq)(b-atq)}{(1-bq)(1-t)}tqF(a, b, t, N)+R_{1, N}(a, b, t),
\end{align}
where
\begin{equation}\label{r1n}
R_{1, N}(a, b, t)=\sum_{j=0}^{N}\frac{(aq)_j}{(bq)_j}t^j\left(1-\frac{(b-atq)q^j}{1-t}\right)+\frac{b-1}{1-t}.
\end{equation}
Clearly, we recover \eqref{f4.1} from \eqref{f4.1finite} upon letting $N\to\infty$ in the latter since $\lim_{N\to\infty}R_{1, N}(a, b, t)=0$.

Fine iterated \eqref{f4.1} to give a proof of the famous Rogers-Fine identity \cite[p.~15]{fine} given by
\begin{align*}
F(a, b; t)=\sum_{n=0}^{\infty}\frac{(aq)_n(atq/b)_n}{(bq)_n(t)_{n+1}}(1-atq^{2n+1})(bt)^nq^{n^2}.
\end{align*}
In the same vein, Andrews and Bell iterated \eqref{f4.1finite} and used the resulting identity along with Euler's method to give a stronger version of the Rogers-Fine identity from which they then deduce the Rogers-Fine identity.

A new finite analogue of $F(a, b; t)$ was studied by the authors of \cite[Section 9]{dems}. It is given for $N\in\mathbb{N}$ by
\begin{equation}\label{finefunctionfin}
F_{N}(a,b;t):=\sum_{n=0}^{N}\left[\begin{matrix} N\\n\end{matrix}\right]\frac{(a q)_{n}(t)_{N-n}(q)_n t^{n}}{(b q)_{n}(t)_N }.
\end{equation}
Clearly, $\lim_{N\to\infty}F_{N}(a,b;t)=F(a, b; t)$.

They naturally encountered this function in their work concerning the theory of the restricted partition function $p(n, N)$. The latter is defined as the number of partitions on $n$ whose parts are less than or equal to $N$. In this study, they derived some properties of $F_N(a, b; t)$ they needed in order to prove their finite analogue of a recent identity of Garvan \cite[Equation (1.3)]{garvan1}. These properties are the partial fraction decomposition of $F_{N}(\a, \b; t)$, namely,
\begin{equation*}
F_{N}(a, b; t)=\frac{(1-t q^{N})(a q)_N}{(b q)_{N}}\sum_{n=0}^{N}\left[\begin{matrix} N\\n\end{matrix}\right]\frac{(b/a)_n(a q)_{N-n}(a q)^n}{(a q)_{N}(1-t q^n)}.
\end{equation*}
This generalizes Fine's partial fraction decomposition of $F(a, b; t)$ \cite[p. 18, Equation (16.3)]{fine}:
\begin{equation}\label{fine16.3}
F(a,b;t)=\frac{(aq)_{\infty}}{(bq)_{\infty}}\sum_{n=0}^{\infty}\frac{(b/a)_n}{(q)_n}\frac{(aq)^n}{1-tq^n}.	
\end{equation}
They also derived a finite analogue of the Rogers-Fine identity \cite[Lemma 9.2]{dems}, that is, for
$\b\neq 0$,
\begin{align*}
F_{N}(a,b;t)=(1- t q^{N})\sum_{n=0}^{N}\left[\begin{matrix} N\\n\end{matrix}\right]\frac{(a q)_n(q)_n\left(\frac{at q}{b}\right)_n(a t q^2)_{N-1}(tb)^nq^{n^2}(1-at q^{2n+1})}{(b q)_n(t)_{n+1}(a t q^2)_{N+n}}.
\end{align*}
These two properties were used to obtain the required finite analogue of Garvan's identity \cite[Theorem 1.3]{dems}. 

At the end of \cite{dems}, it was remarked that it would be worthwhile to develop the theory of $F_{N}(a,b;t)$ considering the enormous impact and applications that the theory of $F(a, b; t)$ has. We set forth such a task with a goal of systematically developing the theory of $F_N(a,b;t)$. 

It is to be noted that Andrews \cite{andfinheine} has already embarked upon this task. Even though his results are expressed as transformations between ${}_3\phi_2$, they can be equivalently written as transformations or functional equations for $F_{N}(a, b; t)$. This rephrasing is done in Section \ref{andrews_results}. Then in Section \ref{new_results}, we present new results on  $F_{N}(a, b; t)$ we have found so far. 

It should be mentioned here that Bowman and Wesley \cite{bowmanwesley} have recently obtained new transformations for Fine's function $F(a, b; t)$ iterating what they call as Fine's ``seed'' identities. A seed identity is an identity where $F(a, b; t)$ gets transformed to $F(aq, b; t)$ or $F(a, bq; t)$ or $F(a, b; tq)$. One of the goals in this paper is obtain the seed identities for $F_N(a, b; t)$.

\section{Some results of Andrews and deductions from them}\label{andrews_results}
In \cite{andfinheine}, Andrews obtained a finite version of Heine's transformation and then derived some corollaries from it. Andrews' transformations are for the ${}_3\phi_{2}$ hypergeometric series 
\begin{align}\label{fine3phi2}
	{}_3\phi_2\begin{bmatrix}
		q^{-N},&q,& aq& ;q,q\\
		bq,&q^{1-N}/t
	\end{bmatrix}=F_N(a,b;t),
\end{align}
as can be easily seen from \eqref{finefunctionfin}.

Thus, we can rephrase these transformations in terms of $F_N(a,b;t)$.
\subsection{An identity of Andrews generalizing Equation (6.3) of \cite{fine}}
From \cite[Corollary 3]{andfinheine},
\begin{align}\label{cor3finheine}
	{}_3\phi_2\begin{bmatrix}
		q^{-N},&abt/c,& b& ;q,q\\
		bt,&bq^{1-N}/c.
	\end{bmatrix}=\frac{(c,t;q)_N}{(c/b, bt; q)_N}{}_3\phi_2\begin{bmatrix}
		q^{-N},&a,& b& ;q,q\\
		c,&q^{1-N}/t
	\end{bmatrix}
\end{align}
If we let $b \rightarrow q$, $t \rightarrow b$, $c \rightarrow tq$, $a \rightarrow atq/b$ in the above identity, this gives
\begin{equation}
F_N(a,b;t)=\frac{(1-tq^N)(1-b)}{(1-bq^N)(1-t)}\sum_{n=0}^{N}\frac{(q^{-N})_n(atq/b)_n(q)_nq^n}{(tq)_n(q^{1-N}/b)_n(q)_n}.
\end{equation}
Using the elementary identities 
\begin{eqnarray}\label{elemid}
	\left(\frac{q^{-N}}{c}\right)_n&=&\frac{(-1)^n(cq^{N-n+1})_nq^{\frac{n(n-1)}{2}}}{c^nq^{Nn}},\\
	\frac{(q^{-N})_n}{\left(\frac{q^{1-N}}{b}\right)_n}&=&\frac{(q)_N(b)_{N-n}b^n}{(q)_{N-n}(b)_Nq^n}\nonumber,
\end{eqnarray}
we are led to
\begin{equation}\label{lemma2fin}
	F_N(a,b;t)=\frac{(1-b)(1-tq^N)}{(1-t)(1-bq^N)}\sum_{n=0}^{N}\begin{bmatrix}
		N\\
		n
	\end{bmatrix}\frac{(atq/b)_n(q)_n(b)_{N-n}b^n}{(tq)_n(b)_N}.
\end{equation}
Note that this identity is a finite analogue of Fine's identity \cite[Equation (6.3)]{fine}:
\begin{equation}\label{lemma2}
	F(a,b;t) = \frac{1-b}{1-t}F(\frac{at}{b},t;b).
\end{equation}
\subsection{Another ${}_3\phi_2$ transformation of Andrews}
Lemma 1 of \cite{andfinheine} gives
\begin{equation}\label{lemma1wala}
	{}_3\phi_2\begin{bmatrix}
		q^{-N}, &a, & b& ;q,q\\
		c, &q^{1-N}/t
	\end{bmatrix}
	= \frac{(at)_n}{(t)_n}{}_3\phi_2\begin{bmatrix}
		q^{-N}, &c/b, & a& ;q, btq^{N}\\
		c, &at
	\end{bmatrix}
\end{equation}
Letting $a \rightarrow q$, $b \rightarrow aq$, $c \rightarrow bq$ in the above identity and using \eqref{fine3phi2}, we get
\begin{align}\label{lemma4wala}
	F_N(a,b;t)&=\frac{(tq)_N}{(t)_N}{}_3\phi_2\begin{bmatrix}
		q^{-N}, &b/a, & q& ;q, atq^{N+1}\\
		bq, &tq
	\end{bmatrix}\nonumber\\
&=\frac{(1-tq^N)}{(1-t)}\sum_{n=0}^{N}(-1)^n
\begin{bmatrix}
	N\\
	n
\end{bmatrix}\frac{(b/a)_n(q)_n(at)^nq^{n(n+1)/2}}{(bq)_n(tq)_n}.
\end{align}
Letting $b=0$ in the above result gives a finite analogue of \cite[Equation (6.1)]{fine}:
\begin{equation}\label{fna0t}
	F_N(a,0;t)= \frac{(1-tq^N)}{(1-t)}\sum_{n=0}^{N}(-1)^n
	\begin{bmatrix}
		N\\
		n
	\end{bmatrix}\frac{(q)_n(at)^nq^{n(n+1)/2}}{(tq)_n}.
\end{equation}
Also, if we let $a\to0$ in \eqref{lemma4wala}, we obtain a finite analogue of \cite[Equation (12.3)]{fine}:
\begin{align}\label{ir1}
	(1-t)F_N(0,b;t)= (1-tq^N)\sum_{n=0}^{N}
	\begin{bmatrix}
		N\\
		n
	\end{bmatrix}\frac{(q)_n(bt)^nq^{n^2}}{(bq)_n(tq)_n}.
\end{align}
If we let $b=t^{-1}= e^{-i\theta}$ in \eqref{ir1}, we get
\begin{align}
	(1-e^{i\theta})F_{N}(0,e^{-i\theta};e^{i\theta}) = (1-e^{i\theta}q^{N})+\sum_{n=1}^{N}\begin{bmatrix}
		N\\
		n
	\end{bmatrix}\frac{(q)_{n}q^{n^2}}{(1-2q\cos{\theta}+q^2) \cdots (1-2q^n\cos{\theta}+q^{2n})},
	\end{align}
which is a finite analogue of \cite[Equation (12.32)]{fine}.

Also, letting $b\to t$ in \eqref{ir1} gives a finite analogue of \cite[Equation (12.33)]{fine}:
 \begin{align*}
 	(1-t)F_N(0,t;t)= (1-tq^N)\sum_{n=0}^{N}
 	\begin{bmatrix}
 		N\\
 		n
 	\end{bmatrix}\frac{(q)_n(t)^{2n}q^{n^2}}{(tq)_n(tq)_n}.
 \end{align*}

Let $r$ and $m$ be positive integers, and replace $q$ by $q^m$, $t$ by $q^r$ in finite analogue of \cite[eqn(12.33)]{fine}.Then using the definition of finite analogue of Fine's function for the left side, We obtain the finite analogue of  \cite[eqn(12.331)]{fine}:
 \begin{align*}
   & \sum_{n=0}^{N}\begin{bmatrix}
      N\\
      n
    \end{bmatrix}_{q^m}\frac{(q^m,q^m)_{n}(q^r)_{N-n}     q^{rn}}{(q^r)_{N}(1-q^r)(1-q^{m+r})\cdots (1-q^{nm+r})}\\
&= (1-q^{N+r})\sum_{n=0}^{N}\begin{bmatrix}
N\\
n
\end{bmatrix}_{q^m}\frac{(q^m,q^m)_{n} q^{mn^2+2rn}}{[(1-q^r)(1-q^{m+r})\cdots(1-q^{nm+r})]^2}
 \end{align*}

\subsection{Andrews' finite Heine transformation in terms of $F_N(a, b; t)$}
Andrews' finite Heine transformation \cite[Theorem 2]{andfinheine} in terms of $F_{N}(a, b; t)$ is given by
\begin{align}
{}_3\phi_2\begin{bmatrix}
	q^{-N},&c/b,& t& ;q,q\\
	at,&q^{1-N}/b
\end{bmatrix}=\frac{(c,t;q)_N}{(b, at; q)_N}	{}_3\phi_2\begin{bmatrix}
		q^{-N},&a,& b& ;q,q\\
		c,&q^{1-N}/t
	\end{bmatrix}.
\end{align}
Let $t \rightarrow q, b \rightarrow t, c \rightarrow atq$ and $a \rightarrow b$ in the above result. Then
\begin{align}\label{withl9}
	F_N(a, b; t)=\frac{(atq,q;q)_N}{(t, bq; q)_N}\sum_{n=0}^{N}\frac{(b)_n(t)_nq^{n}}{(atq)_n(q)_n}.
\end{align}
Letting $N\to\infty$ gives
\begin{align}
	F(a,b;t)= \frac{(atq)_\infty(q)_\infty}{(t)_\infty(bq)_\infty}\sum_{n=0}^{\infty}\frac{(b)_n(t)_nq^n}{(atq)_n(q)_n}.
\end{align}
Letting $b=1$ in \eqref{withl9} gives
\begin{equation}\label{withl9sc}
	F_N(a,1;t)=\frac{(atq)_N}{(t)_N}.
\end{equation}
An interesting result is obtained if we multiply both sides of \eqref{lemma2fin} by $(1-t)$ and then let $t\to1$:
\begin{align}\label{ir}
	\lim_{t\rightarrow{1} }(1-t)F_N(a,b;t)&=\frac{(1-b)(1-tq^N)}{(1-bq^N)}F_N(a/b,1;b)\nonumber\\
	&=\frac{(1-b)(1-tq^N)}{(1-bq^N)}\cdot\frac{(aq)_N}{(b)_N}\nonumber\\
	&=\frac{(1-q^N)(aq)_N}{(bq)_N},
\end{align}
where in the last step, we used \eqref{withl9sc}. Letting $N\to\infty$ in the above result we get Equation (6.31) in Fine's book \cite{fine}:
\begin{equation*}
	\lim_{t\rightarrow{1} }(1-t)F(a,b;t)=\frac{(aq)_\infty}{(bq)_\infty}.
\end{equation*}

If we let $t\to1$ in \eqref{ir1} and compare the right-hand side of the resulting identity with $a=0$ case of \eqref{ir}, we arrive at a finite analogue of \cite[Equation (12.31)]{fine}:
\begin{equation}
	\frac{1}{(bq)_{N}}= (1-q^{N})\sum_{n=0}^{N}\begin{bmatrix}
		N\\
		n
	\end{bmatrix}\frac{(q)_n(b)^nq^{n^2}}{(bq)_n(q)_n}.
\end{equation}
The $b=1$ case of the above identity generalizes \cite[Equation (12.311)]{fine}.

Next, replace $a$ by $b/t$ in \eqref{withl9} and then let $b\to0$ to get
\begin{align}\label{1}
F_N(b/t,0;t)= \frac{(bq)_N(q)_N}{(t)_N}\sum_{n=0}^{N}\frac{(t)_nq^n}{(bq)_n(q)_n},
\end{align}
which is a finite analogue of \cite[Equation (7.3)]{fine}.

Further, if we replace $a$ by $b/t$ in \eqref{fna0t}, we have
\begin{align}\label{2}
	F_N(b/t,0;t)= \frac{(1-tq^N)}{(1-t)}\sum_{n=0}^{N}
	\begin{bmatrix}
		N\\
		n
	\end{bmatrix}\frac{(q)_n(-b)^nq^{n(n+1)/2}}{(tq)_n}.
\end{align}
Letting $t\to0$ in \eqref{1} and \eqref{2} and comparing the right-hand sides, we get
\begin{align}\label{compare}
	\sum_{n=0}^{N}\frac{q^n}{(bq)_{n}(q)_{n}} = \frac{1}{(bq)_{N}(q)_{N}}\sum_{n=0}^{N}
	\begin{bmatrix}
		N\\
		n
	\end{bmatrix}(q)_n(-b)^nq^{n(n+1)/2}.
	\end{align}
This is a finite analogue of \cite[Equation (7.32)]{fine}. If we now put $b=1$ in the above identity, we get a finite analogue of \cite[Equation (7.322)]{fine}.
 
 Letting $b=q^{-1/2}$ in \eqref{compare} and then replacing $q$ by $q^2$, we obtain a finite analogue of \cite[Equation (7.323)]{fine}:
 \begin{align*}
 	\sum_{n=0}^{N}\frac{q^{2n}}{(q)_{2n}} = \frac{1}{(q)_{2N}}\sum_{n=0}^{N}\begin{bmatrix}
 		N\\
 		n
 	\end{bmatrix}_{q^2}(-1)^n(q^2;q^2)_nq^{n^2}.
 \end{align*}
 When $a=0$, \eqref{withl9} reduces to
 \begin{equation}\label{compare1}
 	\frac{(t)_N(bq)_N}{(q)_N}F_N(0,b;t)=\sum_{n=0}^{N}\frac{(b)_n(t)_nq^{n}}{(q)_n},
 \end{equation}
which is a finite analogue of \cite[Equation (11.3)]{fine}.

A finite analogue of \cite[Equation (11.4)]{fine} is obtained if we let $b=t^{-1}$ in \eqref{compare1}:
\begin{align}
\frac{(t)_N(t^{-1}q)_N}{(q)_N}F_N(0,t^{-1};t)=\sum_{n=0}^{N}\frac{(t^{-1})_n(t)_nq^{n}}{(q)_n}.
\end{align}

\section{New results on $F_N(a, b; t)$}\label{new_results}

 Equation (2.2) of \cite{fine} is
 \begin{equation}
 	F(a,b;t) = 1+ \frac{t(1-aq)}{1-bq}F(aq,bq;t).
 \end{equation}
We begin with a finite analogue of this result which takes $(a, b)\to(aq, bq)$.
\begin{theorem}\label{thm1}
	\begin{equation}
F_N(a,b;t)= 1+\frac{t(1-aq)(1-q^{N+1})}{(1-bq)(1-tq^N)}F_N(aq,bq;t).
\end{equation}
\end{theorem}
\begin{proof}
	\begin{align}
		F_N(a,b;t)&=\sum_{n=0}^{N}\begin{bmatrix}
			N\\
			n
		\end{bmatrix}\frac{(aq)_n(t)_{N-n}(q)_nt^n}{(bq)_n(t)_N}\nonumber\\
	&=1+\frac{(1-aq)t}{1-bq}\sum_{n=1}^{N}\begin{bmatrix}
		N\\
		n
	\end{bmatrix}\frac{(q^2a)_{n-1}(t)_{N-n}(q)_nt^{n-1}}{(bq^2)_{n-1}(t)_N}\nonumber\\
&=1+\frac{(1-aq)t}{1-bq}\sum_{n=0}^{N-1}\frac{(q)_{N}(aq^2)_{n}(t)_{N-n-1}t^n}{(q)_{N-n-1}(bq^2)_{n}(t)_N}\nonumber\\
&=1+\frac{(1-aq)t}{(1-bq)}\frac{(1-q^N)}{(1-tq^{N-1})}\sum_{n=0}^{N-1}\frac{(q)_{N-1}(t)_{N-n-1}(aq^2)_nt^n}{(q)_{N-n-1}(t)_{N-1}(bq^2)_n}.
	\end{align}
Using the identity \cite[p.~351, (I.11)]{gasperrahman} 
\begin{equation}\label{gr11}
	\frac{(b,q)_N}{(a,q)_N}\frac{(a,q)_{N-n}}{(b,q)_{N-n}}\frac{a^n}{b^n} = \frac{(q^{1-N}/b,q)_n}{(q^{1-N}/a,q)_n}
\end{equation}
with $b \rightarrow q, a \rightarrow t$ and replacing $N$ by $N-1$, we have
\begin{eqnarray*}
F_N(a,b;t)&=&1+\frac{(1-aq)(1-q^N)t}{(1-bq)(1-tq^{N-1})}\sum_{n=0}^{N-1}\frac{(q^{1-N})_n(aq^2)_nq^n}{(q^{2-N}/t)_n(bq^2)_n}\\
&=&1+\frac{(1-aq)(1-q^N)t}{(1-bq)(1-tq^{N-1})}{}_3\phi_2\begin{bmatrix}
	q^{1-N}, &aq^2, & q& ;q,q\\
	q^{2-N}/t, &bq^2
\end{bmatrix}
\end{eqnarray*}
Replace $N$ by $N+1$ to get
\begin{align}
F_N(a,b;t)&=	1+\frac{(1-aq)(1-q^{N+1})t}{(1-bq)(1-tq^N)}{}_3\phi_2\begin{bmatrix}
		q^{-N}, &aq^2, & q& ;q,q\\
		q^{1-N}/t, &bq^2
	\end{bmatrix}\nonumber\\
&=1+\frac{(1-aq)(1-q^{N+1})t}{(1-bq)(1-tq^N)}F_N(aq,bq,t),
\end{align}
where in the last step we used \eqref{fine3phi2}.
	\end{proof}
Our next result allows us to advance the parameter $t$ in $F_N(a, b; t)$ to $tq$.
\begin{theorem}\label{thm3}
\begin{equation}
	F_N(a,b;t) = \frac{(1-b)(1-tq^{N+1})}{(1-t)(1-bq^{N+1})}+ \frac{(1-q^{N+1})(b-atq)}{(1-bq^{N+1})(1-t)}F_N(a,b;tq).
\end{equation}
\end{theorem}
\begin{proof}
	Using \eqref{lemma2fin}, we have
	\begin{align*}
		F_N(a,b;t)&=\frac{(1-b)(1-tq^N)}{(1-t)(1-bq^N)}\sum_{n=0}^{N}
		\begin{bmatrix}
			N\\
			n
		\end{bmatrix}\frac{(atq/b)_n(q)_n(b)_{N-n}b^n}{(tq)_n(b)_N}\\
		&=\frac{(1-b)(1-tq^N)}{(1-t)(1-bq^N)}+\frac{(1-b)(1-tq^N)}{(1-t)(1-bq^N)}\sum_{n=1}^{N}\begin{bmatrix}
			N\\
			n
		\end{bmatrix}\frac{(atq/b)_n(q)_n(b)_{N-n}b^n}{(tq)_n(b)_N}\\
		&=\frac{(1-b)(1-tq^N)}{(1-t)(1-bq^N)}+\frac{(1-b)(1-tq^N)(1-atq/b)b}{(1-t)(1-bq^N)(1-tq)}\sum_{n=1}^{N}\begin{bmatrix}
			N\\
			n
		\end{bmatrix}\frac{(atq^2/b)_{n-1}(q)_n(b)_{N-n}b^{n-1}}{(tq^2)_{n-1}(b)_N}\\
	&=\frac{(1-b)(1-tq^N)}{(1-t)(1-bq^N)}+\frac{(1-b)(1-tq^N)(1-atq/b)b}{(1-t)(1-bq^N)(1-tq)}\sum_{n=0}^{N-1}\begin{bmatrix}
		N\\
		n+1
	\end{bmatrix}\frac{(atq^2/b)_{n}(q)_{n+1}(b)_{N-n-1}b^{n}}{(tq^2)_{n}(b)_N}\\
	&= \frac{(1-b)(1-tq^N)}{(1-t)(1-bq^N)}+\frac{(1-b)(1-tq^N)(b-atq)(1-q^N)}{(1-t)(1-bq^N)(1-bq^{N-1})(1-tq)}\sum_{n=0}^{N-1}\frac{(q)_{N-1}(atq^2/b)_{n}(b)_{N-n-1}b^{n}}{(q)_{N-n-1}(tq^2)_{n}(b)_{N-1}}.
	\end{align*}
Let $b \rightarrow q, a \rightarrow b$ in \eqref{gr11} , replace $N$ by $N-1$ and substitute the resulting expression in the sum on the right-hand side of the above equation so that
 	\begin{align*}
 	F_N(a,b;t)&=\frac{(1-b)(1-tq^N)}{(1-t)(1-bq^N)}+\frac{(1-b)(1-tq^N)(b-atq)(1-q^N)}{(1-t)(1-bq^N)(1-bq^{N-1})(1-tq)}\sum_{n=0}^{N-1}\frac{(atq^2/b)_{n}(q^{1-N})_n q^{n}}{(q^{2-N}/b)_{n}(tq^2)_{n}}\\
 	&=\frac{(1-b)(1-tq^N)}{(1-t)(1-bq^N)}+\frac{(1-b)(1-tq^N)(b-atq)(1-q^N)}{(1-t)(1-bq^N)(1-bq^{N-1})(1-tq)}{}_3\phi_2\begin{bmatrix}
 		q^{1-N}, &atq^2/b, & q& ;q,q  \\
 		tq^2, &q^{2-N}/b
 	\end{bmatrix}
 \end{align*}
Now replace $a$ by $atq^2/b$, $b$ by $q$, $c$ by $tq^2$, $t$ by $b$ in \eqref{cor3finheine}, then replace $N$ by $N-1$ and use the resulting identity to transform the ${}_3\phi_2$ on the right-hand side of the above equation so that
\begin{align*}
	F_N(a,b;t)&=\frac{(1-b)(1-tq^N)}{(1-t)(1-bq^N)}\nonumber\\
	&\quad+\frac{(1-b)(1-tq^N)(b-atq)(1-q^N)(1-tq)(1-bq^{N-1})}{(1-t)(1-bq^N)(1-bq^{N-1})(1-tq)(1-b)(1-tq^N)}{}_3\phi_2\begin{bmatrix}
		q^{1-N},&aq,& q& ;q,q\\
		bq,&q^{2-N}/tq
	\end{bmatrix}\\
&=\frac{(1-b)(1-tq^N)}{(1-t)(1-bq^N)}+\frac{(1-q^N)(b-atq)(1-bq^{N-1})}{(1-bq^N)(1-bq^{N-1})(1-t)}\sum_{n=0}^{N-1}\frac{(q^{1-N})_n(aq)_n(q)_nq^n}{(bq)_n(q)_n(q^{2-N}/tq)_n}.
\end{align*}
Finally replace $N$ by $N+1$ and use \eqref{elemid} to obtain
\begin{align*}
	F_N(a,b;t)&=\frac{(1-b)(1-tq^{N+1})}{(1-t)(1-bq^{N+1})}+\frac{(1-q^{N+1})(b-atq)(1-bq^{N})}{(1-bq^{N+1})(1-bq^{N})(1-t)}\sum_{n=0}^{N}\frac{(q)_N(aq)_n(q)_n(tq)_{N-n}(tq)^n}{(q)_{N-n}(bq)_n(q)_n(tq)_N}\\
	&=\frac{(1-b)(1-tq^{N+1})}{(1-t)(1-bq^{N+1})}+\frac{(1-q^{N+1})(b-atq)}{(1-bq^{N+1})(1-t)}F_N(a,b;tq).
	\end{align*}
	\end{proof}
	Our next result is a generalization of \cite[Equation (4.3)]{fine}:
\begin{equation}
F(a,b;t) = \frac{1}{1-t}+ \frac{(b-a)tq}{(1-t)(1-bq)}F(a,bq;tq).
\end{equation}	
\begin{theorem}
We have
\begin{align}\label{thm5}
F_N(a,b;t) = \frac{(1-tq^{N+1})}{(1-t)}+ \frac{(1-q^{N+1})(b-a)tq}{(1-bq)(1-t)}F_N(a,bq;tq).
\end{align}
\end{theorem}
	\begin{proof}
	From \eqref{lemma4wala},
\begin{align}
F_N(a,b;t) &= \frac{1-tq^N}{1-t} + \frac{(1-tq^N)}{(1-t)}\sum_{n=1}^{N}(-1)^n
\begin{bmatrix}
N\\
n
\end{bmatrix}\frac{(b/a)_n(q)_n(at)^nq^{n(n+1)/2}}{(bq)_n(tq)_n}\nonumber\\
&= \frac{1-tq^N}{1-t} + \frac{(1-tq^N)(b-a)t}{(1-t)(1-bq)(1-tq)}\sum_{n=0}^{N-1}(-1)^{n}
\frac{(q)_{N}(bq/a)_{n}(at)^{n}q^{(n+2)(n+1)/2}}{(q)_{N-n-1}(bq^2)_{n}(tq^2)_{n}}.
\end{align}	
Now \eqref{elemid} with $c=1/q$ gives
\begin{equation}
(q^{1-N})_{n} = \frac{(-1)^n q^{n(n-1)/2}(q)_{N-1}}{(q)_{N-n-1}q^{(N-1)n}}.
\end{equation}
Hence
\begin{align}\label{jbf}
F_{N}(a, b; t)&= \frac{1-tq^N}{1-t} + \frac{(1-tq^N)(b-a)(1-q^N)tq}{(1-t)(1-bq)(1-tq)}\sum_{n=0}^{N-1}
\frac{(q^{1-N})_{n}q^{Nn+n}(bq/a)_{n}(at)^{n}}{(bq^2)_{n}(tq^2)_{n}}\nonumber\\
&= \frac{1-tq^{N+1}}{1-t} + \frac{(1-tq^{N+1})(b-a)(1-q^{N+1})tq}{(1-t)(1-bq)(1-tq)}\sum_{n=0}^{N}
\frac{(q^{-N})_{n}q^{(N+1)n+n}(bq/a)_{n}(at)^{n}}{(bq^2)_{n}(tq^2)_{n}}.
\end{align}
Now let $b \rightarrow aq, c \rightarrow bq^2, a \rightarrow q$ and replace $t$ by $tq$ in \eqref{lemma1wala} to obtain
\begin{align*}
\sum_{n=0}^{N}
\frac{(q^{-N})_{n}q^{(N+1)n+n}(bq/a)_{n}(at)^{n}}{(bq^2)_{n}(tq^2)_{n}}=\frac{(1-tq)}{(1-tq^{N+1})}F_N(a, bq; tq).
\end{align*}
Finally substitute the above equation in \eqref{jbf} to complete the proof. 
	\end{proof}
In the next result, we transform $F_N(a, b; t)$ to $F_N(a, bq; t)$.
\begin{corollary}\label{thm6}
	We have
	\begin{equation}
	F_{N}(a,b;t)= \frac{1-tq^{N+1}}{1-t}-\frac{(b-a)(1-tq^{N+1})t}{(1-t)(b-at)}+\frac{(b-a)(1-bq^{N+2})t}{(1-bq)(b-at)}F_{N}(a,bq;t).
	\end{equation}
\end{corollary}
\begin{proof}
	Replace $b$ by $bq$ in Theorem \ref{thm3} and then substitute the resulting expression for $F_N(a, bq; tq)$ in \eqref{thm5}.
\end{proof}
Letting $N\to\infty$ in Corollary \ref{thm6} gives \cite[Equation (4.4)]{fine}
\begin{equation*}
	F(a,b;t) = \frac{b}{b-at}+\frac{(b-a)t}{(1-bq)(b-at)}F(a,bq;t).
\end{equation*}
Our next results relates $F_N(a, b; t)$ with $F_N(aq, b;t)$.
\begin{corollary}\label{thm7}
	We have
	\begin{align*}
F_{N}(a,b;t)&=
1-\frac{b(1-aq)(1-q^{N+1})(1-tq^{N+1})}{(1-tq^N)(1-bq^{N+2})(b-aq)}+ \frac{(1-aq)(1-q^{N+1})(b-aqt)}{(1-tq^N)(b-aq)(1-bq^{N+2})}F_{N}(aq,b;t).
	\end{align*}
\end{corollary}
\begin{proof}
	Replace $a$ by $aq$ in Corollary \ref{thm6} and then substitute the resulting expression for $F_N(aq, bq; t)$ in Theorem \ref{thm1} .
\end{proof}
Letting $N\to\infty$ in Corollary \ref{thm7} gives \cite[Equation (4.5)]{fine}
\begin{align*}
F(a,b;t) = -\frac{(1-b)aq}{b-aq} + \frac{(1-aq)(b-atq)}{b-aq}F(aq,b;t).
\end{align*}
Our next theorem transforms $F_N(a, b; t)$ to $F_N(aq,b;tq)$.

\begin{theorem}
	\begin{align*}
		F_{N}(a,b;t)&= 1+\frac{(1-aq)(1-q^{N+1})(1-tq^{N+1})}{(1-tq^N)(1-bq^{N+2})(1-t)}\large\left(t-\frac{(b-aqt)}{(b-aq)}+\frac{(b-aqt)(1-b)}{(b-aq)(1-bq^{N+1})}\large\right)\nonumber\\
		&\quad+ \frac{(1-aq)(1-q^{N+1})^2(b-aqt)(b-atq^2)}{(1-tq^N)(b-aq)(1-bq^{N+2})(1-bq^{N+1})(1-t)}F_{N}(aq,b;tq).
	\end{align*}
\end{theorem}
\begin{proof}
Replace $a$ by $aq$ in Theorem \ref{thm3} and then substitute the resulting expression for $F_N(aq,b;t)$ in Corollary \ref{thm7}.
\end{proof}
The identity in the above theorem leads to the following identity when we let $N\to\infty$:
\begin{align*}
	F(a,b;t)= \large\left(\frac{1-b}{1-t}\large\right)\large\left(1-\frac{(b-atq)}{(b-aq)}aq\large\right) + \frac{(1-aq)(b-atq)(b-atq^2)}{(1-t)(b-aq)}F(aq,b;tq),
\end{align*}
which is Equation (4.6) from \cite{fine}.

Our final result is the transformation between $F_N(a, b; t)$ and $F_N(aq, bq; tq)$.
 \begin{theorem}\label{thm10}
 	We have
 	\begin{align*}
 		F_{N}(a,b;t) = 1 + \frac{(1-aq)(1-q^{N+1})(1-tq^{N+1})t}{(1-tq^N)(1-bq^{N+2})(1-t)} + \frac{(1-aq)(1-q^{N+1})^2(b-atq)tq}{(1-bq)(1-tq^{N})(1-bq^{N+2})(1-t)}F_{N}(aq,bq;tq).
 	\end{align*}
 \end{theorem}
\begin{proof}
Replace $a$ and $b$ in Theorem \ref{thm3} by $aq$ and $bq$ respectively and then substitute the resulting expression for $F_N(aq, bq; t)$ in Theorem \ref{thm1}.
\end{proof}
Letting $N\to\infty$ in Theorem \ref{thm10} gives \cite[Equation (4.1)]{fine}:
\begin{align*}
	F(a,b;t) = \frac{(1-atq)}{(1-t)} + \frac{(1-aq)(b-atq)tq}{(1-bq)(1-t)} F(aq,bq;tq).
\end{align*}

\end{document}